\documentclass[11pt]{amsart}

\usepackage{times}
\usepackage[T1]{fontenc}
\usepackage[utf8]{inputenc}
\usepackage{amssymb}
\usepackage[all]{xy}
\usepackage{amscd,verbatim}
\usepackage[colorlinks,linkcolor=blue,citecolor=blue,urlcolor=red]{hyperref}

\newcommand{\Pic}{\operatorname{Pic}}
\newcommand{\NS}{\operatorname{NS}}

\newcommand{\Spec}{\operatorname{Spec}}

\renewcommand{\lim}{\varprojlim}

\newcommand{\A}{\mathbb{A}}

\newcommand{\G}{\mathbb{G}}

\newcommand{\N}{\mathbf{N}}

\newcommand{\Q}{\mathbf{Q}}

\newcommand{\Z}{\mathbf{Z}}

\newcommand{\Br}{\operatorname{Br}}

\newcommand{\Ker}{\operatorname{Ker}}
\newcommand{\Coker}{\operatorname{Coker}}
\newcommand{\IM}{\operatorname{Im}}

\newcommand{\cl}{\operatorname{cl}}

\renewcommand{\epsilon}{\varepsilon}
\renewcommand{\phi}{\varphi}

\newcommand{\et}{{\operatorname{\acute{e}t}}}

\newcommand{\rd}{{\operatorname{red}}}

\newcommand{\inj}{\hookrightarrow}

\newcommand{\iso}{\overset{\sim}{\longrightarrow}}
\newcommand{\by}{\xrightarrow}

\newtheorem{thm}{Theorem}[section]
\newtheorem{prop}[thm]{Proposition}
\newtheorem{cor}[thm]{Corollary}
\newtheorem{lemma}[thm]{Lemma}
\theoremstyle{definition}

\theoremstyle{remark}
\newtheorem{rk}[thm]{Remark}
\newtheorem{rks}[thm]{Remarks}
\newtheorem{qn}[thm]{Question}

\newcounter{spec}
{\end{list}}%
\numberwithin{equation}{section}
\begin{document}
\title{An approach to the Tate conjecture for surfaces over a finite field}
\author{Bruno Kahn}
\address{FJ-LMI -- IRL2025 CNRS\\
Graduate School of Mathematical Sciences\\
The University of Tokyo\\
3-8-1 Komaba, Meguro, 153-8914\\
Tokyo, Japan}
\email{bruno.kahn@imj-prg.fr}
\date{\today}
\begin{abstract} We give a reformuation of the Tate conjecture for a surface over a finite field in terms of suitable affine open subsets. We then present three attempts to prove this reformulation, each of them falling short.  Interestingly, the last two are related to techniques used in proofs of Gersten's conjecture.
\end{abstract}
\maketitle


\section*{Introduction}

Since the Tate conjecture in codimension $1$ has been reduced to surfaces over the prime field \cite{morrow,ambrosi,kahness}, it is tempting to try to prove it in such case (but see the philosophical comment below). In this note and in the spirit of \cite{kahness}, I present a reformulation for a surface $X$ over a finite field $k$ in terms of affine open subsets $U\subset X$ which verify $\Pic(U)=0$. I then present three ideas, so far unsuccessful, to prove this reformulation.

I thank Olivier Wittenberg, Hiroyasu Miyazaki and Takeshi Saito for their patience in reading or listening to explanations of initial versions of this work.

\subsection*{Philosophical comment} There are so many possible paths towards a proof of the Tate conjecture (for divisors) that there is no guarantee that the present strategy is right. Some positive evidence is that the arguments below strongly use the hypotheses that $\dim X=2$ and that $k$ is finite: M. Artin's theorem on the cohomological dimension of affine schemes, and the procyclicity of the absolute Galois group. Yet no proof comes out so far. Perhaps out of frustration, I thought it worthwhile to make these ideas public nevertheless.

\subsection*{Notation} Let $l$ be a prime number. For any scheme $S$ on which $l$ is invertible and any integers $(i,j)\in \N\times \Z$,  we write $H^i(S,j):=H^i_\et(S,\Q_l/\Z_l(j))$ and $\Br_l(S):=H^2_\et(S,\G_m)\{l\}$ for the $l$-primary part of the cohomological Brauer group.

Let $k$ be a field, $k_s$ a separable closure of $k$ and $G=Gal(k_s/k)$. If $V$ is a $k$-variety, we write $V_s:=V\otimes_k k_s$.

\section{The reformulation}

\begin{thm}\label{t1}  Suppose $k$ finite, and let $X$ be a smooth projective surface over $k$; assume that $G$ acts trivially on $\NS(X_s)
$. Then the following are equivalent:
\begin{enumerate}
\item The Tate conjecture holds for $X$.
\item For any affine open $U\subset X$ such that $\Pic(U)=0$, one has $H^3(U,1)=0$.
\item For any affine open $U\subset X$ such that $\Pic(U)=0$ and any smooth irreducible divisor $Z\subset U$, the map $H^3(U,1)\to H^3(U-Z,1)$ is injective. 
\end{enumerate}
\end{thm}

\begin{rks} a) We can always reduce to the case where $G$ acts trivially on $\NS(X_s)$ after passing to a finite extension of $k$, since $\NS(X_s)$ is finitely generated. This is sufficient for the Tate conjecture.\\
b) Such $U$'s as in (2) and (3) exist because $\Pic(X)$ is finitely generated.\\
c) Quasi-affine is not sufficient in (2): by purity, $H^3(\A^2-\{0\},1)=H^0(k,-1)$, which is $\ne 0$ in general.
\end{rks}

For the proof, we may assume $X$ geometrically connected. Of course, (2) is a step from (1) to (3). We shall prove successively (1) $\iff$ (2) and (2) $\iff$ (3).

\subsection{The Brauer group of $X$}
The following proposition works for any smooth projective variety $X$ over any finitely generated field $k$.

\begin{prop}\label{p0} The Tate conjecture holds for $X$ in codimension $1$ if and only if $\Br_l(X_s)^G$ is finite.
\end{prop}

\begin{proof} The Kummer exact sequences yield a short exact sequence
\[0\to \NS(X_s)\otimes \Q_l\to H^2(X_s,\Q_l(1))\to V_l(\Br_l(X))\to 0\]
where $V_l$ denotes the rational Tate module. Take Galois cohomology and observe that 
\[H^1(G,\NS(X_s)\otimes \Q_l)=H^1(G,\NS(X_s))\otimes \Q_l=0,\] 
so the Tate conjecture holds if and only if $V_l(\Br_l(X))^G=0$. This is equivalent to the finiteness of $\Br_l(X_s)^G$, because $\Br_l(X_s)$ is of cofinite type (as a quotient of $H^2(X_s,1)$).
\end{proof}

\subsection{Cohomology of curves} We go back to the case where $k$ is finite and $X$ is a surface.

\begin{prop}\label{p1} Let $Z$ be a smooth curve over $k$. Then $H^1(Z_s,0)$ is divisible and $H^1(Z_s,0)^G$ is finite, of order bounded independently of $l$.
\end{prop}

\begin{proof} This easily follows from Weil's Riemann hypothesis applied to the smooth completion $\bar Z$ of $Z$ and the Gysin exact sequence relating $H^1(\bar Z_s,0)$ and $H^1(Z_s,0)$. The proof shows that the bound only depends on the field of constants, the genus of $\bar Z$ and the divisor at infinity.
\end{proof}

\subsection{The Brauer group of $U$}

Let $U\subseteq X$ be an open subset and $Z$ its (reduced) closed complement: we have a short exact sequence
\[0\to \Br_l(X_s)\to \Br_l(U_s) \to \bigoplus_{x\in Z\cap X^{(1)}} H^1((Z'_x)_s,0)\]
where, for all $x\in Z\cap X^{(1)}$, $Z'_x$ is the intersection of the smooth locus of $Z$ with its irreducible  component corresponding to $x$. Proposition \ref{p1} then yields

\begin{prop}
The Tate conjecture holds for $X$ if and only if for any $U$ open in $X$, $\Br_l(U_s)^G$ is finite, and also if and only if there exists such a $U$. \qed
\end{prop}

\begin{prop} $\Br_l(U_s)^G$ is finite if and only if $\Br_l(U_s)_G$ is finite.
\end{prop}

\begin{proof} This is true for any $G$-module of cofinite type, because $G$ is procyclic: an endomorphism of a finite-dimensional $\Q_l$-vector space is injective if and only if it is surjective.
\end{proof}

\subsection{Passing from $\Br_l(U_s)_G$ to $H^3(U,1)$}

If $\NS(U_s)$ is torsion, then $H^2(U_s,1)\allowbreak\iso \Br_l(U_s)$, hence the  Hochschild-Serre spectral sequence yields a short exact sequence:
\[0\to \Br_l(U_s)_G\to H^3(U,1)\to H^3(U_s,1)^G\to 0.\]

\begin{lemma} If moreover $U$ is affine, the first map of this sequence is an isomorphism of divisible groups.
\end{lemma}

\begin{proof}
This follows from M. Artin's theorem that $cd_l(U_s)=2$ \cite[Cor. 3.2]{sga4}, applied twice: first, the right hand term of the sequence is $0$, and then $\Br_l(U_s)$ is divisible because $H^3(U_s,\mu_l)=0$..
\end{proof}

\subsection{The condition $\Pic(U)=0$}

\begin{lemma} Suppose that $G$ acts trivially on $\NS(X_s)$. Then the map $\Pic(U_s)^G\to \NS(U_s)$ is surjective. We have the equivalences: $\Pic(U)$ is torsion $\iff$ $\Pic(U_s)$ is torsion $\iff$ $\NS(U_s)$ is torsion. In particular, $\Pic(U)=0$ $\Rightarrow$ $\NS(U_s)$ is torsion.
\end{lemma}

\begin{proof} If $G$ acts trivially on $\NS(X_s)$, it acts trivially on its quotient $\NS(U_s)$; moreover, $\Pic^0(X_s)\to \Pic^0(U_s)$ is surjective (compare the exact sequences of \cite[Prop. 1.8 and Ex. 10.3.4]{fulton}), hence $\Pic^0(U_s)$ is torsion since $k$ is finite. It follows that the map
\[\Pic(U_s)^G\to \NS(U_s)^G=\NS(U_s)\]
has torsion kernel; it is surjective because the next term $H^1(G,\Pic^0(U_s))$ is $0$. Indeed, since $cd_l(G)=1$, this group is a quotient of $H^1(G,\Pic^0(X_s))$ which is $0$ for the same reason as in the proof of Proposition \ref{p1}.
 
Finally, the groups $\Ker(\Pic(U)\to \Pic(U_s)^G)$ and $\Coker(\Pic(U)\to \Pic(U_s)^G)$ are torsion by a transfer argument. The conclusion follows easily.
\end{proof}

\subsection{Proof of $(1)\iff (2)$}

This follows from what has been proven so far. Namely: the Tate conjecture holds for $X$ $\iff$ $\Br_l(X_s)^G$ finite $\iff$ $\Br_l(U_s)_G$ finite $\iff$ $\Br_l(U_s)_G=0$ (because it is divisible) $\iff$ $H^3(U,1)=0$.

\subsection{From $(2)$ to $(3)$}

We still assume that $G$ acts trivially on $\NS(X_s)$ (a blanket assumption now).

\begin{prop}\label{A} Condition (2) is equivalent to the following: for any affine open $U\subset X$ such that $\Pic(U)=0$ and any open $V\subseteq U$, the map $H^3(U,1)\to H^3(V,1)$ is injective.
\end{prop}

\begin{proof}
If the statement is true, then $H^3(U,1)$ injects into $H^3(K,1)$ where $K=k(X)=k(U)$; but this group is $0$ by \cite[prop. 4]{purete}.
\end{proof}

\begin{prop}\label{B} The condition of Proposition \ref{A} is equivalent to the same statement, but with $Z:=(U-V)_\rd$ irreducible of dimension $1$.
\end{prop}

\begin{proof} Let $D_1,\dots, D_n$ be the irreductible components of codimension $1$ of $Z$. For $0\le i\le n$,  define $U_i$ inductively as $U_{i-1}\setminus D_i$, with $U_0=U$. We have a chain of open subsets
\[U\supset U_1\supset\dots U_n\supseteq V\]
where each $U_i$ is affine since $D_i$ is principal, $\Pic(U_i)=0$, and $U_n-V$ is of codimension $\ge 2$ in $U_n$. By the hypothesis, $H^3(U_{i},1)\inj H^3(U_{i+1},1)$ for all $i$,  and also  $H^3(U_{n},1)\inj H^3(V,1)$ by cohomological purity.
\end{proof}

\begin{prop}\label{C} The condition of Proposition \ref{B} is equivalent to the same statement, but with $Z$ smooth.
\end{prop}

\begin{proof} Let 
$\bar Z$ be the closure of $Z$ in $X$ and $F$ be its singular locus. By Poonen \cite[Cor. 3.4]{poonen}, there exists a smooth projective curve $C_0\subset X$  containing $F$; a fortiori, $C=C_0\cap U$ is smooth. Apply the hypothesis to $(U,C)$ and then to $(U-C, Z\setminus C)$ (note that $Z\setminus C$ is smooth): we get that the composition
\[H^3(U,1)\to H^3(U-C,1)\to H^3(U-(C\cup Z),1)\] 
is injective. A fortiori, $H^3(U,1)\to H^3(U-Z,1)$ is injective.
\end{proof}

This concludes the proof of (2) $\iff$ (3), hence of Theorem \ref{t1}.


\section{Going further}

Let $(U,Z)$ be as in Condition  (3) of Theorem \ref{t1}, and let $V=U-Z$.  We have the Gysin exact sequence
\begin{equation}\label{eq2}H^2(V,1)\by{\partial} H^1(Z,0)\by{i_*} H^3(U,1)\by{j^*} H^3(V,1) 
\end{equation}
which yields two reformulations of this condition: the vanishing of $i_*$, and the surjectivity of $\partial$. The following proposition cuts down the size of the image of $i_*$.

\begin{prop}\label{p3} In \eqref{eq2},\\ 
a) the image of $\partial$  contains the image of  $i^*:H^1(U,0)\to H^1(Z,0)$.\\
b)  $i_*$ factors through the finite group $H^1(Z_s,0)^G$.\\
c) $i_*=0$ (hence $j^*$ is injective)  for $l\ge l_0$, where $l_0$ is a  prime number depending on $Z$.
\end{prop}

\begin{proof} a) Let $f\in \Gamma(U,\G_a)$ be an equation of $Z$ in $U$. Then $f$ is inversible on $V$. Let  $(f)\in H^1(V,\Z_l(1))$ be its Kummer class: the composition
\[H^1(U,0)\by{j^*}H^1(V,0)\by{\cup (f)}H^2(V,1)\by{\partial} H^1(Z,0)\]
equals $i^*$: this follows from the definition of the purity isomorphism  \cite[2.1]{cycle}.

b) Let $k_Z$ be the field of constants of $Z$. We have a commutative diagram of exact sequences
\begin{equation}\label{eq7.1}\begin{CD}
0@>>> H^0(Z_s,0)_G@>>> H^1(Z,0)@>>> H^1(Z_s,0)^G@>>> 0\\
&&@Ai^* AA @Ai^* AA @Ai^* AA\\
0@>>> H^0(U_s,0)_G@>>> H^1(U,0)@>>> H^1(U_s,0)^G@>>> 0
\end{CD}
\end{equation}
where the left vertical arrow is multiplication by $[k_Z:k]$ in $\Q_l/\Z_l$, hence surjective. By a), $\IM\partial\supseteq \IM H^0(Z_s,0)_G$.

c) follows from Proposition \ref{p1}.
\end{proof}

\section{Three approaches}

\subsection{The first idea fails} 

It would be to use the fact that the Tate conjecture is independent of $l$ \cite[Prop. 4.3]{tate}. Unfortunately, $l_0$ is not  (a priori) bounded independently of $Z$ in Proposition \ref{p3} c).

\begin{rk} On the other hand, the corank of $H^3(U,1)$ is independent of $l$, as is more generally the characteristic polynomial of Frobenius acting on $H^2(U_s,\Q_l)$. By an Euler-Poincaré characteristic argument, it suffices to see the same thing for $H^0$ and $H^1$. This is trivial for $H^0$; for $H^1$, by semi-purity we have an exact sequence
\[0\to H^1(X_s,\Q_l) \to H^1(U_s,\Q_l)\to \bigoplus_{x\in X_s^{(1)}\setminus U_s} \Q_l(-1)\by{\cl_l} H^2(X_s,\Q_l)\] 
so it suffices to see the independence of $l$ for $\Ker(\cl_l)$. But $\cl_l$ factors through $\cl\otimes \Q_l(-1)$, where $\cl$ is the divisor map  $\bigoplus\limits_{x\in X_s^{(1)}\setminus U_s} \Z\to \NS(X_s)$.

If one replaces $H^2(U_s,\Q_l)$ by $H^2_c(U_s,\Q_l)$, this computation shows by Poincaré duality that the corresponding inverse characteristic polynomial has integer coefficients.
\end{rk}

\subsection{Second idea: from above} 

It is inspired by Gillet's proof of Gersten's conjecture for dvr's for $K$-theory with finite coefficients \cite{gillet}, abstracted in \cite[app. B]{cthk}. I divide it in three parts: the first two work, but not the last.

\subsubsection{First part}  It is motivated by Gabber's rigidity theorem \cite{gabber}:

\begin{thm}\label{t3}  $(A_h,I)$ Henselian pair, $U_h = \Spec(A_h)$, $Z = \Spec(A_h/I)$, $i : Z \inj U_h$ the closed immersion. Then for any torsion abelian étale sheaf $F$ on  $U_h$ and for all $q>0$,  $H^q(U_h,F)\by{i^*} H^q(Z,F)$ is bijective.
\end{thm}

Coming back to our pair $(U,Z)$: recall that a \emph{Nisnevich neighbourhood} of $Z\inj U$ is a Cartesian square
\begin{equation}\label{eq8}
\begin{CD}
V_1@>j_1>> U_1\\
@Vp VV @Vq VV\\
V@>j>> U
\end{CD}
\end{equation}
with $q$ étale and $q^{-1}(Z)\iso Z$. The henselisation $(U_h,Z)$ of $(U,Z)$ is the filtering colimit of such neighbourhoods.

\begin{cor}\label{c1} There exists $(U_1,q)$ such that $H^1(U_1,0)\to H^1(Z,0)$ is surjective.
\end{cor}

\begin{proof} By Theorem \ref{t3}, the homomorphism $H^1((U_h)_s,0)^G\to H^1(Z_s,0)^G$ is bijective. As the right hand side is finite (Proposition \ref{p1}), there exists $U_1$ such that $H^1((U_1)_s,0)^G\to H^1(Z_s,0)^G$ is surjective. Then $H^1(U_1,0)\to H^1(Z,0)$
is also surjective, by the snake lemma applied to (the analog of) \eqref{eq7.1}: use the surjectivity of the left vertical map as in the proof of Proposition \ref{p3} b).
\end{proof}

By Proposition \ref{p3} a), Corollary \ref{c1} implies that $\partial_1:H^2(V_1,1)\to H^1(Z,0)$ is surjective. Unfortunately, this is not sufficient: how do we go down? There is no push-forward for $p$.

\subsubsection{Second part} 
To correct this problem, consider the normalisation $\bar U_1$ of $U$ in $q$, enriching \eqref{eq8} into a more complicated diagram
\begin{equation}\label{eq9}
\vcenter{
\xymatrix{
&\bar V_1\ar[rr]^{\bar j_1}\ar[ddl]^(.3){\bar p}&& \bar U_1\ar[ddl]^(.3){\bar q}&\bar Z\ar[l]_{\bar i_1}\ar[ddl]^(.3){\bar r}\\
V_1\ar[rr]^{j_1}\ar[d]^p\ar[ur]^{j'}&& U_1\ar[d]^q\ar[ur]^{j''}&Z\ar[l]_{i_1}\ar[d]_=\ar[ur]^{u}\\
V\ar[rr]^j&& U&Z\ar[l]_i
}}
\end{equation}
where $j''$ is an open immersion, $\bar q$ is finite (since $q$ is étale), $\bar V_1=V\times_{U} \bar U_1$, $\bar Z=\bar U_1-\bar V_1$ and the other arrows follow. In particular, $j'$ and $u$ are also open immersions, $\bar Z$ is closed and $\bar p$, $\bar r$ are also finite.

Some observations on this diagram:

\begin{itemize}
\item Finite $\Rightarrow$ affine, hence $\bar U_1$,  $\bar V_1$ are affine. All vertices of \eqref{eq9} are affine.
\item In particular, the closed immersion $\bar i_1$ is purely of codimension $1$ \cite[cor. 21.12.7]{EGA4}.
\item Since $\bar r$ is separated, the open immersion $u$ is also closed, hence $\bar Z=Z\coprod T$ for some other closed subset $T$ (of pure codimension $1$).
\item Since $\bar U_1$ and $\bar V_1$ are normal surfaces, $\bar p$ and $\bar q$ are \emph{flat} \cite[Ex. 8.2.15]{liu}.
\end{itemize}

Since $\bar p,\bar q$ are finite and flat, trace maps are available in étale cohomology \cite[Th. 6.2.3]{sga4-17} and we have a commutative diagram
\begin{equation}\label{eq10}
\vcenter{
\xymatrix{
H^2(V,1)\ar[d]^{\partial}& H^2(\bar V_1,1)\ar[l]_{\bar p_*}\ar[d]^{\bar\partial_1}\ar[r]^{{j'}^*}& H^2(V_1,1)\ar[d]^{\partial_1}\\
H^3_{Z}(U,1)& H^3_{Z}(\bar U_1,1)\oplus H^3_{T}(\bar U_1,1)\ar[l]_(0.6){\bar q_*}\ar[r]^(0.6){{j''}^*}& H^3_{Z}(U_1,1)\\
&H^1(Z,0)\ar[ul]^\sim\ar[ur]_\sim\ar[u]^a
}}
\end{equation}
where $\partial_1$ is surjective as seen above, and $a$ is the isomorphism on the first summand defined by excision ($H^3_{Z}(\bar U_1,1)\iso H^3_{Z}(U_1,1)$),  and $0$ on the second. The left square commutes e.g. by proper (finite)  base change. This implies:

\begin{cor} \label{l11}  $\IM\partial\supseteq \IM(\bar q_*\circ \bar \partial_1)$.\qed 
\end{cor}

\subsubsection{Third part} \label{3.2.3}
If we could show that $\IM \bar\partial_1\supseteq \IM a$, we would win. We would like to use the surjectivity of $\partial_1$, but it is not sufficient. Analysing Gillet's arguments, things would work if
\begin{enumerate}
\item  the composition
\[\bar i_{1,Z}^*:H^1(\bar U_1,0)\by{{j''}^*} H^1(U_1,0)\by{i_1^*} H^1(Z,0)\]
is surjective, and
\item $\exists$ $f_1\in \Gamma(\bar U_1,\G_a)$ such that $Z$ is principal of equation $f_1$ in $\bar U_1$, and $f_1\equiv 1\pmod{T}$.
\end{enumerate}

(2) looks very expensive (see \ref{3.2.4} below), but maybe (1) can be achieved by enlarging $U_1$. Specifically, we may ask the following

\begin{qn} Let $\bar U_h$ be the normalisation of $U$ in $U_h\to U$. Is the composition
\begin{equation}\label{eq1}
H^1(\bar U_h,0)\to H^1(U_h,0)\iso H^1(Z,0)
\end{equation}
still surjective?
\end{qn}

Note that this composition is injective by \cite[lemme 3.6 and rem. 3.7]{sga2}, and that surjectivity holds anyway for $l$ large enough, because it holds for $i^*$ by Proposition \ref{p3} c).

\subsubsection{Sone negative ``evidence'' in dimension $1$}\label{3.2.4} Instead of a surface, consider a smooth affine $k$-curve $U$ such that $\Pic(U)=0$, and let $Z\in U$ be a closed point. There is a trivial way to make $H^1(U,0)\to H^1(Z,0)$ surjective: extend scalars from $k$ to $k(Z)$ so that $U_{k(Z)}$ acquires a rational point above $Z$, and shrink $U$ in order to throw away the extra points above $Z$. 

The analogue to Condition (1) in \ref{3.2.3} is obviously verified. However, even though $\Pic(U)=0$, $\Pic(U_{k(Z)})$ will not vanish in general because $\Pic(U_{k_s})$ is infinite as soon as the smooth completion of $U$ has positive genus. Thus Condition (2) may well fail. There is probably an explicit example with $U$ an affine elliptic curve.

\subsection{Third idea:  from below} 

This idea is inspired by Gabber's geometric presentation lemma, used to prove Gersten's conjecture (\cite{gabber2}, \cite[th. 3.1.1]{cthk}).

Suppose that we can construct an ``ante-Nisnevich neighbourhood'' of $i$:
\begin{equation}\label{eq11}
\vcenter{
\xymatrix{
Z\ar[dr]_{i_1}\ar[r]^i&U\ar[d]^v\\
&U_1}
}
\end{equation}
where $i_1$ is a closed immersion, $v$ is a Nisnevich neighbourhood of $Z$ and $U_1$ is affine open in a smooth projective surface for which Tate's conjecture is known. Then  $(i_1)_*=0$, hence $i_*=v^*(i_1)_*=0$ by functoriality of the Gysin maps. 

In fact, ``Nisnevich neighbourhood'' is not necessary: by the functoriality of Gysin morphisms, $v$ may be any morphism such that
\begin{equation}\label{eq12}
Z= v^{-1}(\overline{v(Z)}),\quad Z\iso \overline{v(Z)}
\end{equation}
(scheme-theoretically). Moreover, $v(Z)$ is constructible by Chevalley's theorem, but $Z$ is a curve, hence $v(Z)$ is open in its closure.

Gabber's lemma achieves this with $U_1=\A^2$, but up to an open subset. A version over finite fields was given by Hogadi-Kulkarni \cite[Lemma 2.4 and Rem. 1.3]{hogadi}:

\begin{prop}
There exists a morphism $v:U\to \A^2$ and an open subset $W\subseteq \A^2$ such that
\begin{enumerate}
\item $v_{|v^{-1}(W)}$ is \'etale
\item $Z\cap v^{-1}(W)\by{v} W$ is a closed immersion.
\end{enumerate}
\end{prop}

But we cannot afford to ``lose'' a closed subset in $U$ (of codimension $2$, à la rigueur\dots): we would start running into circles. Thus, let us look at $v$ on the whole of $U$. The second condition of \eqref{eq12} is achieved\footnote{Essentially, because $\overline{v(Z)}$ might be singular; but this concerns only a finite set of closed points.}, but not the first: there can be extra components -- and there will be in general, because $v$ has generic degree $>1$ unless it is birational\dots\ see \ref{3.3.1} below for a definite counterexample in the analogous dimension $1$ case.

This problem  is similar to the one encountered in the second idea!

\subsubsection{Sone negative ``evidence'' in dimension $1$}\label{3.3.1} Let $(U,Z)$ be as in \S \ref{3.2.4} and let  $X$ be the smooth projective completion of $U$, assumed to be of genus $>0$.

\begin{prop} Let $N=|(X-U)(k_s)|$. If $\deg(Z)>N$, there is no diagram \eqref{eq11} verifying \eqref{eq12} with $d=\deg(v)>1$ (generic degree).
\end{prop}

\begin{proof} Suppose that such a diagram exists. Then $v$ induces a ramified covering $\bar v:X\to X_1$ of degree $d$, where $X_1$ is the smooth completion of $U_1$. As in \ref{3.2.4}, deploy the situation by extending scalars to $k(Z)$: then $Z$ is replaced by $\delta$ rational points  $Z_i$, where $\delta=\deg(Z)$. By hypothesis, one has $Z_i=v^{-1}(v(Z_i))$ for all $i$. But since $d>1$, $\bar v^{-1}(v(Z_i))-\{Z_i\}$ is non-empty and contained in $(X-U)(k_s)$. When $i$ varies, these sets are disjoint. Contradiction.
\end{proof}

On the other hand, if $Z$ is a rational point, the construction works by using the vanishing of $\Pic(U)$: if $f\in k(U)$ is an equation of $Z$, it suffices to project $U$ to $\A^1_{k}$ via $f$.

I cannot create counter-examples such as \ref{3.2.4} and \ref{3.3.1} in dimension $2$...

\end{document}